\documentclass[mn,a4paper,fleqn]{w-art}
\newcommand{\Ueberschrift}{Lifting Galois sections along torsors}
\newcommand{\Kurztitel}{Lifting Galois sections}
\usepackage{times,cite,w-thm}
\usepackage{amssymb} 
\usepackage{mathrsfs}
\usepackage[all]{xy}
\usepackage[OT2, T1]{fontenc}
\usepackage[pdfpagelabels]{hyperref}
 \hypersetup{
   pdftitle={},
   pdfauthor={},
   pdfsubject={},
  pdfkeywords={},
  colorlinks=true,
  breaklinks=true,
  bookmarksopen=true,
  bookmarksnumbered=true,
  pdfpagemode=UseOutlines,
  plainpages=false}

\DeclareMathOperator{\rH}{H}

\DeclareMathOperator{\rK}{K}

\DeclareMathOperator{\Rm}{m}

\newcommand{\bA}{{\mathbb A}}

\newcommand{\bG}{{\mathbb G}}

\newcommand{\bN}{{\mathbb N}}

\newcommand{\bP}{{\mathbb P}}
\newcommand{\bQ}{{\mathbb Q}}
\newcommand{\bR}{{\mathbb R}}

\newcommand{\bT}{{\mathbb T}}

\newcommand{\bZ}{{\mathbb Z}}

\newcommand{\cF}{{\mathscr F}}

\newcommand{\dL}{{\mathcal L}}
\newcommand{\dM}{{\mathcal M}}

\newcommand{\dO}{{\mathcal O}}

\DeclareSymbolFont{cyrletters}{OT2}{wncyr}{m}{n}
\DeclareMathSymbol{\Sha}{\mathalpha}{cyrletters}{"58}

\newcommand{\surj}{\twoheadrightarrow} 
\newcommand{\inj}{\hookrightarrow}
\DeclareMathOperator{\id}{id}
\DeclareMathOperator{\pr}{pr}

\DeclareMathOperator{\BIsom}{\mathbf{Isom}}

\DeclareMathOperator{\Hom}{Hom}

\DeclareMathOperator{\End}{End}

\DeclareMathOperator{\im}{im}

\DeclareMathOperator{\Et}{Et}
\DeclareMathOperator{\FEt}{FEt}

\DeclareMathOperator{\Char}{char} 

\DeclareMathOperator{\Spec}{Spec}

\DeclareMathOperator{\Pic}{Pic}
\DeclareMathOperator{\CPic}{\mathcal{P}\it{ic}}

\newcommand{\Gm}{\bG_{\rm m}}

\DeclareMathOperator{\B}{B}

\DeclareMathOperator{\Br}{Br}
\newcommand{\inv}{{\rm inv}}

\DeclareMathOperator{\h}{H}
\DeclareMathOperator{\R}{R}

\DeclareMathOperator{\ind}{ind}

\DeclareMathOperator{\Gal}{Gal}

\newcommand{\ph}{\varphi}

\newcommand{\et}{\text{\rm \'et}} 

\newcommand{\tors}{{\rm tors}}

\newcommand{\hZ}{{\hat{\bZ}}}

\newtheorem*{thm*}{Theorem}
\begin{document}

\DOIsuffix{theDOIsuffix}
\Volume{XYZ}
\Month{MM}
\Year{YYYY}
\pagespan{1}{}
\Receiveddate{XXXX}
\Reviseddate{XXXX}
\Accepteddate{XXXX}
\Dateposted{XXXX}

\title[\Kurztitel]{\Ueberschrift} 

\author[N. Borne]{Niels Borne\inst{1}
  \footnote{Corresponding author\quad E-mail:~\textsf{Niels.Borne@math.univ-lille1.fr}}}
\address[\inst{1}]{U.M.R. CNRS 8524, U.F.R. de Math\'ematiques, Universit\'e Lille 1, 59 655 Villeneuve d'Ascq C\'edex, France}

\author[M. Emsalem]{Michel Emsalem\inst{1}}

\author[J. Stix]{Jakob Stix\inst{2}}
\address[\inst{2}]{Institut f\"ur Mathematik, 
Goethe--Universit\"at Frankfurt, Ro\-bert-Mayer-Stra\ss e~6--8,
60325~Frankfurt am Main, Germany}

%
%

\thanks{Niels Borne and Michel Emsalem were supported in part by the Labex CEMPI (ANR-11-LABX-0007-01) and Anr ARIVAF (ANR-10-JCJC 0107)}

\date{\today} 

\begin{abstract} 
The \emph{cuspidalization conjecture}, which is a consequence of Grothendieck's \emph{section conjecture}, asserts that for any smooth hyperbolic curve $X$ over a finitely generated field $k$ of characteristic $0$ and any non empty Zariski open $U \subset X$, every section of  $\pi _1 (X, \bar x) \to \Gal_k$
lifts to a section of $\pi _1 (U,\bar x) \to  \Gal_k$. We consider in this article the problem of lifting Galois sections to the intermediate quotient $ \pi_1^{cc}(U)$
introduced by Mochizuki \cite{mochizuki:abscusp}. We show that when $k = \bQ$ and $D=X\setminus U$ is an union of torsion sub-packets every Galois section actually lifts to $ \pi_1^{cc}(U)$. One of the main tools in the proof is the construction of torus torsors $F_D$ and $E_D$ over $X$ and the geometric interpretation $ \pi_1^{cc}(U) \simeq \pi _1 (F_D)$.
\end{abstract}

\maketitle
\bibliographystyle{mn} 


\section{Introduction}

Let  $\bar k$ be a fixed separable closure of an arbitrary field $k$, and let $\Gal_k =\Gal(\bar k / k)$ be  the absolute Galois group of $k$.  For a variety $X/k$,   let $X_{\bar k} =X\times_k\bar k$ be the base change of $X$ to $\bar k$.


\subsection{Section conjecture and cuspidalization conjecture} 
\label{sub:section_cuspidalization}

The \'etale fundamental group $\pi_1(X, \bar{x})$ of $X$ with base point 
$\bar x$ fits into an 
exact sequence, denoted by $\pi_1(X/k)$:
\begin{equation} \label{eq:pi1ext}
1 \to \pi_1(X_{\bar k}, \bar{x}) \to \pi_1(X,\bar{x}) \to \Gal_k \to 1.
\end{equation}
By functoriality, any rational point $x\in X(k)$ induces a section $s_X(x):\Gal_k\to \pi_1(X,\bar{x})$, well defined in the set $\mathscr S_{\pi_1(X/k)}$ of sections of \eqref{eq:pi1ext} up to conjugation by an element of $\pi_1(X_{\bar k}, \bar{x})$.

In his $1983$ anabelian letter to Faltings (see \cite{grothendieck:brief}), Grothendieck formulated the 
\emph{section conjecture}, 
namely: if $k$ is a field of finite type over $\mathbb Q$, and $X$ is a smooth proper curve of genus at least $2$ over $k$, then the section map
\[s_X: X(k) \rightarrow \mathscr S_{\pi_1(X/k)}  \]
is a bijective correspondence. He also explained how to deduce the fact that $s_X$ is injective from the Mordell-Weil theorem.

Grothendieck's original vision also contains a version of the section conjecture for a nonempty open subset $U$ of $X$. It is however easy to see that the direct analogue fails: the section map \(s_U: U(k) \rightarrow \mathscr S_{\pi_1(U/k)}  \) is still injective, but not surjective in general. Indeed for each rational point $x \in (X\backslash U)(k)$ at infinity, one can consider the local extension $\pi_1(U_x/k)$ associated to a punctured formal neighbourhood $U_x = \Spec(\hat{\dO}_{X,x}) \setminus \{x\}$ of $x$ in $U$. This gives rise to a non-empty packet  $\mathscr S_{\pi_1(U_x/k)}\subset \mathscr S_{\pi_1(U/k)}$ of \emph{cuspidal sections}. The generalized section conjecture now asserts that if $U$ is a hyperbolic smooth curve over a field $k$ of finite type over $\mathbb Q$, then sections in $\mathscr S_{\pi_1(U/k)}$ either come from rational points of $U$, or are cuspidal.

The direct observation that the local extensions $\pi_1(U_x/k)$ split has the following striking consequence: if the section conjecture holds for $X$, then the map $\mathscr S_{\pi_1(U/k)}\to \mathscr S_{\pi_1(X/k)}$ is surjective. This consequence of the section conjecture we are fond of calling the \emph{cuspidalization conjecture}. 

Besides being a test for the section conjecture, the cuspidalization conjecture is also a significant part of various strategies to prove the section conjecture itself. For example,  by working with a fixed and  very specific open subset $U$, namely the complement 
\[
U=X\backslash X(k)
\]
of the set of rational points of $X$, the cuspidalization conjecture enables to reduce the section conjecture to the statement: if a smooth curve has no rational points, every section is cuspidal. This version of the section conjecture should be more tractable, considering the fact that cuspidal sections have been characterized among all sections by Nakamura (see \cite{nakamura:rigidityb,stix:habil}) as those sections that cyclotomically normalize an inertia subgroup.  We emphasize that $X(k) \subseteq X$ is a union of torsion sub-packets in the sense to be described below.

\subsection{Main result} 
\label{sub:main_result}

Our main result concerns lifting of sections in $\mathscr S_{\pi_1(X/k)}$ to the \emph{cuspidally central fundamental group} $\pi_1^{cc}(U)$, which is an intermediate quotient
\[ \pi_1(U) \twoheadrightarrow \pi_1^{cc}(U) \twoheadrightarrow  \pi_1(X) \]
introduced by Mochizuki (see \cite{mochizuki:abscusp}, we recall the original profinite definition in \S\ref{sec:comparewithpicc}. 
Sa\"idi  considered $\pi_1^{cc}(U)$ in the context of the cuspidalization conjecture, see \cite{saidi:goodsections}), and we will compare below how Sa\"idi's approach relates to ours.

The main novelty of this article is that one can relate the geometry of the complement $D=X\backslash U$ to this lifting problem. Following Baker-Poonen (see \cite{baker-poonen:torsionpackets}), we will say that a reduced divisor $D$ is a \emph{torsion sub-packet} if geometrically the difference of any two points in the support of $D$ is torsion in the Picard group (see 
Definition~\ref{defi:torsionpacket}). We can now state:

 \begin{thm*}[Theorem~\ref{thm:mainresult?}] \label{thmA}
 Let $X/\bQ$ be a smooth projective curve of positive genus, let $D \subset X$ be a union of  torsion sub-packets, and set $U=X\backslash D$. Then  every Galois section $s : \Gal_\bQ \to \pi_1(X)$  lifts to a section $\Gal_\bQ \to \pi_1^{cc}(U)$.
 \end{thm*}

We want to stress that Theorem~\ref{thm:mainresult?} seems to be the first \textbf{unconditional} result of lifting to $\pi_1^{cc}(U)$: we make no assumption on the section, 
in contrast to the setup in \cite{saidi:goodsections}, which defines "good" sections and shows that these are precisely the ones that can be lifted. 
 
An unconditional but weaker result had been obtained by the first two authors in \cite{borne-emsalem-critere}.

Let us now explain the main ideas of the proof of our result.
The first step consists of introducing a torsor $F_D$ over $X$ under the torus $T_D=\R _{D/k}\mathbb G_m$  whose \'etale fundamental group $ \pi_1(F_D)$  identifies with $\pi_1^{cc}(U)$  (see Proposition~\ref{cc}). 

The torsor $F_D$ is itself obtained as a $\mathbb G_m$-torsor over an intermediate torsor $E_D$ over $X$ under the torus $S_D=T_D/\mathbb G_m$. In \S\ref{sec:torsorstodivisor}, we prove that $D$ is a torsion sub-packet if and only if $E_D$ is torsion (see Proposition~\ref{thm:crit-torsion-torsors}), a fact that will be crucial in the last step of the proof.
 
We then study the obstruction of lifting a Galois section along a general torsor $E$ over $X$ under a torus $T$. We show in 
Proposition~\ref{thm:pi1extfortorsor} that, if $X$ has $\pi_2^\et(X,\bar x) = 0$,
the morphism $E\to X$ gives rise to a natural fibration short exact sequence, denoted by $\pi_1(E/X)$. The next step identifies the class of the extension $\pi_1(E/X)$ with the arithmetic first Chern class $c_1(E)$ of $E$ (see Proposition~\ref{prop:inj+coincidence}).

Returning to the specific situation of Theorem~\ref{thm:mainresult?}, and given a section $s$ of $X$, the last and most delicate step consists of killing the obstruction $s^*(c_1(F_D))$. This is the aim of \S\ref{sec:lifting}. Besides the fact that the torsor $E_D$ is torsion, we use crucially that the relative Brauer group $\Br(X/k)=\ker(\Br(k)\to \Br(X))$ vanishes when $k=\mathbb Q$ in presence of a section, a statement proven by the third author in \cite{stix:periodindex}. It is this result that limits the scope of Theorem~\ref{thm:mainresult?} to the base field $k= \bQ$, a limitation that may be seen as an indication that the section conjecture could be more accessible for the base field $\bQ$ (like many other results in arithmetic).

\subsection{Notation and conventions}

The notation $\R_{S'/S}(-)$ denotes Weil restriction of scalars along the finite flat map $S' \to S$ which is implicit in the notation.

By  $X_{S'}$ we will denote the base change $X \times_S S'$ of an $S$ scheme $X$ by $S' \to S$. However, we would like to direct the kind reader's attention to the following exceptions. The notation $S_D, T_D$ (resp.\ $E_D,F_D$) introduced in 
Section~\S\ref{sec:torsorstodivisor} denote a certain torus (resp.\ torus torsor) associated to a divisor $D$.

When talking about torsors, one has to fix a topology in which the torsors trivialize locally. This will be the \'etale topology always without further mention.

For a map $X \to S$, denote by $\Pic_{X/S}$ the relative Picard scheme. 
The isomorphism class of a line bundle $\dL$ will be denoted by $[\dL]$.

\smallskip

\subsection*{Acknowledgments} 
The authors would like to thank Amaury Thuillier and Angelo Vistoli for discussions and the ENS Lyon for hospitality during a visit of the third author. We would also like to thank the referees for their careful reading of our text and their suggestions that helped to improve the exposition. 

\smallskip

\section{Torsors associated to a divisor on a curve}
\label{sec:torsorstodivisor}

Let $X$ be a smooth curve  defined over $S= \Spec (k)$.
\subsection{Definition of torsors \texorpdfstring{$E_D$}{ED} and   \texorpdfstring{$F_D$}{FD}} 
\label{func_def_tors}
Let $\Delta : X\to X\times_S X$ be the diagonal embedding. This is a Cartier divisor, defining an invertible sheaf $\mathcal O_{X\times_S X} (\Delta)$. Let $D$ be an effective, \'etale Cartier divisor on $X$. We will always assume that $D\neq 0$. 

\begin{defn}\label{def_torsors}
We define torsors
\[
F_D=\R_{D\times_S X/X}\left(\BIsom_{D\times_S X}(\mathcal O_{X\times_S X} (\Delta)_{|D\times_S X},\mathcal O_{D\times_S X})  \right)   \to X
\]
under the torus $T_D= \R_{D/S}(\Gm)$,
\[
E_D=F_D/\Gm \to X
\]
under the quotient torus $S_D=T_D/\Gm$,  the cokernel of the adjoint map $\Gm \to \R_{D/S} \Gm$.
\end{defn}

\begin{rem}
The torsor $E_D$ can also be introduced in a natural way via Picard schemes. Namely, let $X_D$ be the curve obtained from $X$ by pinching $D$ into a single rational point. One can check that the torsor $E_D\to X$ is the pullback of $\Pic_{X_D/k} \to \Pic_{X/k}$ via the morphism $X\to\Pic_{X/k}$  given by $\Delta$. The construction of $F_D$ is somewhat more subtle: one remarks that the morphism $X\to\Pic_{X/k}$ (thus resp. $E_D\to \Pic_{X_D/k}$) factors through the Picard \emph{stack} $\CPic_{X/k}$ (resp. $\CPic_{X_D/k}$). Since the canonical rational point of $X_D$ gives rise to a morphism $\CPic_{X_D/k}\to \B\mathbb G_m$, one gets by composition a morphism $E_D\to \B\mathbb G_m$ and one verifies easily that the corresponding $\mathbb G_m$-torsor is $F_D\to E_D$.
\end{rem}


\subsection{Torsion sub-packets and torsion criterion}
\label{subsec:tors-pack}

Let $k$ is a field of characteristic $0$, and $X$ be a smooth, proper, geometrically connected curve over $S=\Spec (k)$.

\begin{defn} \label{defi:torsionpacket}
A (reduced) effective divisor $D$ is a \emph{torsion sub-packet} if any degree $0$ divisor on $X_{\bar k}$ with support in $D_{\bar k}$ is torsion.
This means that in the sense of Baker and Poonen \cite{baker-poonen:torsionpackets} the divisor $D_{\bar k}$ is contained in a single torsion packet of $X_{\bar k}$.
\end{defn}

\begin{rem}
\begin{enumerate}
	\item Any rational point defines a torsion sub-packet.
	\item If $D$ is a torsion sub-packet, then any degree $0$ divisor on $X$ with support in $D$ is torsion, because the map $\Pic X \inj \Pic X_{\bar k}$  is injective.
	\item However, it is not true that if any degree $0$ divisor on $X$ with support in $D$ is torsion, then $D$ is a torsion sub-packet. For instance, if $D$ is irreducible, then any degree $0$ divisor on $X$ with support in $D$ is even trivial. But $D$ does not need to be a torsion sub-packet. Indeed, according to \cite{baker-poonen:torsionpackets}, Corollary 3, if $\Char k=0$ and $X$ is of genus at least $2$, the size of torsion packets is bounded. It is thus enough to choose $D$ of degree strictly larger than this size, which is possible if $k$ has an infinite absolute Galois group, to get a counterexample. 
\end{enumerate}
\end{rem}

\begin{prop}
	\label{thm:crit-torsion-torsors}
Let $X$ be a smooth, proper, geometrically connected curve over $S=\Spec (k)$.
 The torsor $E_D \to X$ is torsion if and only if the \'etale divisor $D \subseteq X$ is a torsion sub-packet.
\end{prop}

\begin{proof}
The Hochschild-Serre spectral sequence gives the following exact sequence:
 \[ 0 \rightarrow \h^1(k,S_D) \rightarrow \h^1(X,S_D) \rightarrow \h^1(X_{\bar k},S_D)^{\Gal_k} \]
 Moreover, the first group is torsion by Lemma~\ref{lem:h^1(S_D)} below, thus $E_D$ is torsion if and only if $\left(E_D\right)_{\bar k}$ is torsion. As the formation of $E_D$ also commutes  with base change, we can assume that $k$ is algebraically closed. The group of characters $\Hom(S_D, \Gm)$ of $S_D$ is the group  \(  \left(\oplus_{D( k)} \bZ\right)^{\sum = 0} \) of divisors  $\delta$ of degree $0$ with support in $D$. Since it is free of finite rank, it is enough to show that for such divisor $\delta$, the $\mathbb G_m$-torsor obtained by reduction of the structure group of $E_D$ is torsion. But this $\mathbb G_m$-torsor is nothing else than $\BIsom(\dO_X,\dO_X(\mathcal \delta))$, and the result follows. 
\end{proof}

\begin{defn}
\label{defi:relBR}
(1) For any $k$-scheme $Y$, we define the relative Brauer group  $\Br(Y/k)$ as the kernel of the natural morphism 
$\Br(k)\to \Br(Y)$.

(2) For any $k$-scheme $Y$ of finite type, we define the index 
\[
\ind(Y) =\gcd\left\{ \deg(y) \ ; \  y \text{ is a closed point of } Y\right\},
\]
where $\deg(y)= [\kappa(y):k]$ is the degree of the residue field extension at $y$.
\end{defn}

\begin{lem}
\label{lem:h^1(S_D)}
We have 
\[ 
\h^1(k,S_D)=\Br(D/k),
\]
and, in particular, $\h^1(k,S_D)$ is torsion and killed by $\ind(D)$. 
\end{lem}
\begin{proof}
This follows from the Galois cohomology exact sequence associated to 
\begin{equation} \label{eq:definitionofSD}
0 \to \Gm \to T_D = \R_{D/S}(\Gm) \to S_D \to 0,
\end{equation}
Shapiro's Lemma and Hilbert's Theorem 90. For an extension $k'/k$ and $x \in D(k')$, the pullback
\[
x^\ast :  \big(\R_{D/S}(\Gm)\big)_{k'} \to \bG_{\Rm, k'}
\]
splits \eqref{eq:definitionofSD} and so 
\[
\Br(D/k) \inj \Br(k'/k)
\]
which by a corestriction argument is killed by $[k':k]$. Hence $\Br(D/k)$ is killed by $\ind(D)$.
\end{proof}

\section{Extensions of fundamental groups associated to torsors} \label{sec:pi}

\subsection{Extensions of fundamental groups associated to fibrations}
Following \cite{artin-mazur:etalehomotopy} and \cite{Friedlander:etalehomotopysimplicialschemes}  a noetherian scheme $X$ with a geometric point $\bar x$ has  \'etale homotopy groups 
\[
\pi_i^\et(X, \bar x)
\]
for $i \geq 0$. If $X$ is geometrically unibranch, then $\pi_1^\et(X,\bar x) = \pi_1(X,\bar x)$. It is well known that smooth geometrically connected curves $X$ over any field $k$ such that  $X_{\overline k}\not\simeq \bP^1_{\overline k}$  are algebraic $\rK(\pi,1)$ spaces, which in particular means that $\pi_2^\et(X, \bar x) = 0$.

\begin{prop} 
\label{thm:pi1extfortorsor}
Let $k$ be a field of characteristic $0$. 
Let $X/k$ be a geometrically connected and geometrically unibranch variety, let $T/k$ be a torus and $E \to X$ a torsor under $T$. Let $\bar y$ be a geometric point of $E$ with image $\bar x \in X$. 

If $\pi_2^\et(X,\bar x) = 0$, then the sequence
\begin{equation} \label{eq:pi1fortorsors}
1 \to \pi_1(E_{\bar x}, \bar y) \to \pi_1(E,\bar y) \to \pi_1(X,\bar x) \to 1
\end{equation}
is exact.
\end{prop}

\begin{proof}
Because of the exact sequence \eqref{eq:pi1ext} for $E$ and $X$, we may assume that $k$ is algebraically closed. Then $T \simeq \Gm^r$ is a trivial torus, and $E \to X$ is Zariski locally isomorphic to the trivial $\Gm^r$-torsor. Therefore $E \to X$ is a geometric fibration in the sense of  \cite{Friedlander:etalehomotopysimplicialschemes} Definition 11.4. By \cite{Friedlander:etalehomotopysimplicialschemes} Theorem 11.5 we have an exact homotopy sequence
\[
\pi^\et_2(X,\bar x) \to \pi^\et_1(E_{\bar x}, \bar y) \to \pi^\et_1(E,\bar y) \to \pi^\et_1(X,\bar x) \to \pi^\et_0(E_{\bar{x}},\bar y)
\]
and the claim follows from $\pi_2^\et(X,\bar x) = 0$ and $\pi^\et_0(E_{\bar{x}},\bar y) = 0$.
\end{proof}

\begin{rem}
The vanishing assumption for $\pi^\et_2(X, \bar x)$ is indeed important. As an example, we consider an algebraically closed field $\bar k$,   some $n \geq 1$, and the $\Gm$-torsor 
\[
\bA_{\bar k}^{n+1} \setminus \{0\} \to \bP_{\bar k}^n
\]
associated to the line bundle $\dO(1)$. Then by Zariski-Nagata purity of the branch locus we have $\pi_1(\bA_{\bar k}^{n+1} \setminus \{0\}) = 0$ and the sequence
\[
 \pi_1(\bG_{\Rm, \bar k}) \to \pi_1(\bA_{\bar k}^{n+1} \setminus \{0\})  \to \pi_1(\bP_{\bar  k}^n) \to 1
\]
 is not injective on the left. Indeed, the group $\pi_2^\et(\bP_{\bar  k}^n) \not= 0$ does not vanish.
 \end{rem}

\begin{rem}  \label{rmk:pi1modulestructure}
In the situation of Proposition~\ref{thm:pi1extfortorsor}, the choice of $\bar y$ defines  an isomorphism $T_{\bar x} \simeq E_{\bar x}$ by translation, and  the group 
\[
\pi_1(T_{\bar x},1) \simeq  \pi_1(E_{\bar x}, \bar y)
\]
is the fundamental group of an algebraic group (in characteristic $0$, see \cite{stix:habil} \S13.1) and hence abelian (so that we can neglect base points). The conjugation action of $\pi_1(E,\bar y)$ on $\pi_1(E_{\bar x}, \bar y)$ thus defines a $\pi_1(X,\bar x)$-module structure on $\pi_1(E_{\bar x}, \bar y)$.
\end{rem}

\begin{prop} \label{prop:pi1modulestructure}
Let $X$ be a geometrically connected and geometrically unibranch variety over $S = \Spec(k)$ of characteristic $0$, and assume that $\pi_2^\et(X,\bar x) = 0$. 
Let $T/S$ be a torus and $E \to X$ a torsor under $T$.
Then the $\pi_1(X,\bar x)$-action on $\pi_1(E_{\bar x}, \bar y)$ factors over the projection $\pi_1(X,\bar x) \to \Gal_k$ and translation is an isomorphism
\[
\tau \ : \ \bT(T) \simeq \pi_1(T_{\bar k},1)  = \pi_1(T_{X,\bar x},1)  \simeq  \pi_1(E_{\bar x}, \bar y)
\]
where $\bT(T)$ denotes the Tate module of $T$.
\end{prop}
\begin{proof}
The $\pi_1(X,\bar x)$-module structure on $\pi_1(T_{\bar k},1) =  \pi_1(T_{X,\bar x},1) $ associated to $T_X \to X$ comes by functoriality from the action associated to $T \to \Spec(k)$. The action thus factors through the projection $\pi_1(X,\bar x) \to \Gal_k$. The identification  $\bT(T) \simeq \pi_1(T_{\bar k},1)$ as $\Gal_k$-modules is classical, see for example \cite{stix:habil} Section~\S13.1. 

We consider the isomorphism $\Phi: T_X \times_X E \to E \times_X E$ over  $X$ defined by multiplication $T \times_X E \to E$  and second projection $T \times_X E \to E$. By Proposition~\ref{thm:pi1extfortorsor}, we obtain an isomorphism of extensions
\[
\xymatrix{
	0 \ar[r]& \pi_1\big((T_X \times_X E)_{\bar x},(1,\bar y)\big) \ar[d]^{\simeq}  \ar[r] & \pi_1(T_X \times_X E, (1,\bar y)) \ar[d]^{\pi_1(\Phi)}  \ar[r] & \pi_1(X,\bar x) \ar[d]^{\id} \ar[r] & 0 \\
	0 \ar[r]& \pi_1\big((E \times_X E)_{\bar x},(\bar y,\bar y)\big) \ar[r] & \pi_1(E \times_X E, (\bar y, \bar y)) \ar[r] & \pi_1(X,\bar x) \ar[r] & 0
}
\]
Since $\Phi$ is a map over $E$ via second projection, we obtain an isomorphism of $\pi_1(X,\bar x)$-modules
\begin{align*}
\pi_1(T_{\bar k} ,1) & = \ker\Big(\pi_1\big((T_X \times_X E)_{\bar x},(1,\bar y)\big) \xrightarrow{\pi_1(\pr_2)} \pi_1(E_{\bar x}, \bar y) \Big) \\
 & \simeq \ker\Big(\pi_1\big((E \times_X E)_{\bar x},(\bar y,\bar y)\big) \xrightarrow{\pi_1(\pr_2)} \pi_1(E_{\bar x}, \bar y) \Big)  = \pi_1(E_{\bar x}, \bar y)
\end{align*}
 induced by $\Phi$. This completes the proof.
\end{proof}

\subsection{Comparison with the maximal cuspidally central quotient}
\label{sec:comparewithpicc}

The aim of this paragraph is to mention the following interpretation of  $\pi_1(F_D,\bar y)$. 
Let $U = X \setminus D$ be the complement of the support of the divisor and set 
\[
N={\rm Ker} (\pi _1 (U,\bar x) \to \pi _1 (X,\bar x)).
\]
Recall the notion of the maximal \textbf{cuspidally central} quotient $\pi_1^{cc}(U,\bar x)$ due to Mochizuki \cite{mochizuki:abscusp} Definition 1.5(i):  the biggest quotient $\pi _1 ^{cc} (U,\bar x) = \pi_1 (U,\bar x)/N_{cc}$ by a normal subgroup $N_{cc} \subseteq N$ such that one gets an exact sequence
$$
1 \to N^{cc} \to \pi _1 ^{cc} (U,\bar x) \to \pi _1 (X,\bar x) \to 1
$$
where $N^{cc} = N/N_{cc} $ is abelian, and the action of $\pi _1 (X_{\bar k},\bar x) $ by conjugation on $N^{cc}$ is trivial.

\begin{prop}
\label{cc}
The canonical lift $U \to F_D$ of the inclusion $U \subset X$ induces an isomorphism
$$
\pi_1^{cc}(U,\bar x) \simeq \pi_1(F_D, \bar x).
$$ 
\end{prop}

\begin{proof}
If $D$ is totally split, this follows from \cite{mochizuki:abscusp}, Proposition 1.8 iii). Since the claim is of geometric nature, the general case follows. 
\end{proof}


\section{Obstructions to lifting Galois sections along torsors under tori}
\label{sec:obstruction}

In this section, we fix a smooth, proper, geometrically connected curve $X$ of genus at least~$1$ over a field $k$ of characteristic $0$, and $T/k$ a torus. Let $E\to X$ be a torsor under the torus $T$. Then $E$ is also geometrically connected and therefore defines a fundamental exact sequence analogous to \eqref{eq:pi1ext}. The issue we want to address is: does a given section $s:\Gal_k \to \pi_1(X,\bar x)$ of \eqref{eq:pi1ext} lift to a section $\Gal_k \to \pi_1(E,\bar y)$?

\subsection{Arithmetic first Chern class}

We introduce in this paragraph the notion of arithmetic first Chern class of a torsor under a torus. 
The relevance of this notion for anabelian issues has been pointed out by Mochizuki in the case of line bundles 
(see \cite{mochizuki-local} Definition 0.3). Our definition is the straightforward generalization. 

The most logical way to proceed is to use Jannsen's cohomology (see \cite{jannsen:continuous}), that is, cohomology defined on the category of inverse systems of \'etale sheaves of abelian groups on the small \'etale site of the base scheme $X$. In this context, the Kummer short exact sequence takes the following form, for a pair of integers $(m,n)$ such that $m|n$:

\begin{equation}
	\label{eq:Kummer}
\xymatrix{
	0 \ar[r]& \ar[d]^{\frac{n}{m}\cdot} T[n]\ar[r] & T\ar[d]^{\frac{n}{m}\cdot} \ar[r]^{n\cdot} & T \ar[d]^{\id} \ar[r] & 0 \\
	0 \ar[r]& T[m]\ar[r] & T \ar[r]^{m\cdot} & T \ar[r] & 0
}
\end{equation}
Let us define $\mathbb T(T)$ (resp. $(T,\frac{n}{m}\cdot$)) as the inverse system given by the left (resp. middle) column of diagram \eqref{eq:Kummer}. Note that the Jannsen cohomology of the right column is the usual \'etale cohomology of $T$.

\begin{defn}
\label{def:ari_first_Chern}
	The \emph{first arithmetic Chern class} of a $T$-torsor $E \to X$ is the image of its class by the coboundary morphism
	\[
	c_1: \h^1(X,T) \to \h^2(X,\mathbb T(T))
	\]
	in the Jannsen cohomology long exact sequence associated to the short exact sequence \eqref{eq:Kummer}. 	
\end{defn}

\begin{prop}
\label{prop:Chern-short-exact}  
\ 
\begin{enumerate}
\item  \label{item1prop:continuousH2}
The following morphism is an isomorphism:
\[
\h^2(X,\mathbb T(T))\to \varprojlim_{n\in \bN}\h^2(X, T[n]).
\]
\item  \label{item2prop:exactsequencecontinuousH2}
For a torus $T$ over a field $k$, the following sequence is exact:
\[ 
0 \to \h^1(k,T) \xrightarrow{c_1} \h^2\left(k,\mathbb T(T)\right) \to \mathbb T\left(\h^2(k,T)\right)\to 0.
\]
\end{enumerate}
\end{prop}
\begin{proof}
\eqref{item1prop:continuousH2} 
 The short exact $\varprojlim$-sequence (see \cite{jannsen:continuous}, (3.1))
 \begin{equation*}
 0 \to \varprojlim_{n\in \mathbb N} \!^1 \ \h^{1}(X,T[n]) \to \h^{2}\left(X,\mathbb{T}(T)\right) \to  \varprojlim_{n\in \bN}\h^2(X, T[n]) \to 0
 \end{equation*}
reduces assertion (1) to the vanishing of the $\varprojlim^1$-term.  The Kummer sequence provides  the exact sequence
 \[
 0 \to T(X)/nT(X) \to \rH^1(X,T[n]) \to \rH^1(X,T)[n] \to 0,
 \]
 which by Mittag--Leffler and $\varprojlim_{n \in \bN}^2 = 0$ leads to an exact sequence
 \[
 0 = \varprojlim_{n\in \mathbb N} \! ^1 \ T(X)/nT(X) \to \varprojlim_{n\in \mathbb N} \! ^1 \ \h^{1}(X,T[n]) \xrightarrow{\sim} \varprojlim_{n\in \mathbb N} \! ^1 \  \rH^1(X,T)[n] \to 0.
 \]
 The restriction to $n$-torsion of the short exact sequence of low degree terms of the Hochschild--Serre spectral sequence for $X \to \Spec(k)$ and coefficients $T$ yields exactness of 
 \[
 0 \to \rH^1(k,T)[n] \to \rH^1(X,T)[n] \to \rH^1(X_{\bar k},T)[n]^{\Gal_k}.
 \]
 By Hilbert's Theorem 90, the groups $\rH^1(k,T)[n]$ have finite exponent with a bound for the exponent independent of $n$. Thus the system $\bT(\rH^1(k,T))$ is Mittag--Leffler zero (\cite{jannsen:continuous}, 1.10). Since $T_{\bar k} \cong \Gm^d$ for $d = \dim T$ we find non-canonically 
 \[
 \rH^1(X_{\bar k},T) \cong \Pic(X_{\bar k}) ^d
 \]
 which has finite $n$-torsion. Thus the projective system $( \rH^1(X,T)[n] )$ is an extension of a system of finite levels 
 \[
 \left(\im(\rH^1(X,T)[n] \to \rH^1(X_{\bar k},T)[n]^{\Gal_k}\right)
 \]
 by a Mittag--Leffler zero system. Therefore its $\varprojlim^1_{n \in \bN}$ vanishes and the proof is complete.

\eqref{item2prop:exactsequencecontinuousH2} The Kummer sequences for $T$ induces exact sequences
\[
0 \to \rH^1(k,T)/n \rH^1(k,T) \to \rH^2(k,T[n]) \to \rH^2(k, T)[n] \to 0.
\]
Since by Hilbert's Thereom 90 the group $\rH^1(k,T)$ has bounded exponent, we have 
\begin{align*}
\varprojlim_{n\in \mathbb N} \rH^1(k,T)/n \rH^1(k,T)  & = \rH^1(k,T) \\
\varprojlim_{n\in \mathbb N} \!^1 \ \rH^1(k,T)/n \rH^1(k,T)  & = 0.
\end{align*}
Therefore assertion \eqref{item2prop:exactsequencecontinuousH2} follows by passing to the limit and 
\eqref{item1prop:continuousH2}  for $X = \Spec(k)$.
\end{proof}

\begin{rem}	
\label{rmk:modnchernsuffice}
	Because the morphism $\h^2(X,\mathbb T(T))\to \varprojlim_{n\in \bN}\h^2(X, T[n])$ is an isomorphism,
 it will be often sufficient to consider only Chern classes modulo $n$, denoted by $c_1(E)_n$, and defined as the image of $c_1(E)$ by the morphism $\h^2(X,\mathbb T(T))\to \h^2(X, T[n])$. 
\end{rem}

\subsection{Class of the fibration}

Let  $h:E \to X$ be  a $T$-torsor, and consider the exact sequence
\begin{equation}
\label{eq:hom-exact-seq}
1 \to \pi_1(T_{\bar k},1) \to \pi_1(E,\bar y) \to \pi_1(X,\bar x) \to 1
\end{equation}
from Proposition~\ref{thm:pi1extfortorsor} where we use the isomorphism $\pi_1(T_{\bar k},1) \simeq \pi_1(E_{\bar x},\bar y)$ of 
Proposition~\ref{prop:pi1modulestructure}. This yields an abelian cohomology class 
\[
\pi_1(E/X) \in \h^2(\pi_1(X,\bar x),\pi_1(T_{\bar k},1)).
\]
The section $s$ lifts to (a Galois section of) $E$ if and only if $s^*(\pi_1(E/X))=0$ in $\h^2(\Gal_k,\pi_1(T_{\bar k},1))$. We are therefore interested in an explicit description of the class $\pi_1(E/X)$.

\begin{prop}
\label{prop:inj+coincidence} 
Let  $X$ be a geometrically connected and geometrically unibranch variety over a field $k$ of characteristic $0$. Assume that $\pi_2^\et(X,\bar x) = 0$. 

Then 
the morphism $\h^2(\pi_1(X,\bar x),\pi_1(T_{\bar k},1)) \to \h^2(X,\mathbb T(T))$ is an isomorphism and sends $\pi_1(E/X)$ to $c_1(E)$.	
\end{prop}

\begin{proof}
Let $\tilde{X}_\et$ be the universal covering of the \'etale homotopy type $X_\et$ with respect to the geometric point $\bar x \in X$. Then for any locally constant constructible torsion sheaf $\cF$ we have isomorphisms
\begin{align*}
\rH^0(\tilde{X}_\et, \cF) & \simeq  \cF_{\bar x}, \\
\rH^1(\tilde{X}_\et, \cF) & = 0, \\
\rH^2(\tilde{X}_\et, \cF) & \simeq \Hom(\pi_2^\et(X,\bar x), \cF_{\bar x}).
\end{align*}
The exact sequence of low degree terms for the Leray spectral sequence for $\tilde{X}_\et \to X_\et$ yields
\[
 \rH^1(\pi_1(X,\bar x),\cF_{\bar x}) \simeq \rH^1(X,\cF)
\]
and the exact sequence
\[
0 \to \rH^2(\pi_1(X,\bar x),\cF_{\bar x}) \to \rH^2(X,\cF) \to \Hom_{\pi_1(X,\bar x)}( \pi_2^\et(X,\bar x), \cF_{\bar x}).
\]
For $\cF = T[n]$, using the short exact $\varprojlim$-sequence (see \cite{jannsen:continuous}, (3.1)) shows that 
\[
\h^2(\pi_1(X,\bar x),\pi_1(T_{\bar k},1))  \simeq \h^2(\pi_1(X,\bar x),\mathbb{T}(T)) \to \h^2(X,\mathbb T(T))
\]
is an isomorphism.

We now prove the second claim based on ideas by Mochizuki \cite{mochizuki:topics} Lemma~4.4+5. 
Denote by $\Et(X)$  the small \'etale site of $X$, and by  $\FEt(X)$ the finite \'etale site of $X$. 
The natural morphism $\gamma: \Et(X)\to \FEt(X)$ induces for each sheaf  $\cF$ on $\FEt(X)$ and the corresponding  locally constant sheaf $\gamma^\ast \cF$ on $\Et(X)$ and each $i \geq 0$ a morphism 

\[\rH^i(\pi_1(X,\bar x),\cF_{\bar x}) \to \rH^i(X,\gamma^\ast \cF) \] 
via the usual identification of sheaves on $\FEt(X)$ and representations of $\pi_1(X,\bar x)$.
Note that this coincides with the edge map of the spectral sequence used above.

Now the fibration $E \to X$ gives rise to a commutative diagram of sites:
\[
	\xymatrix{
	\Et(E)\ar[r] \ar[d]& \ar[d]\FEt(E)\\	
	\Et(X)\ar[r] & \FEt(X)
	}
\] 
This diagram induces in turn a morphism of Leray-Serre spectral sequences associated to the sheaf $\mathbb T(T)$ corresponding to the $\pi_1(E)$-representation $\pi_1(T_{\bar k},1)$ (we omit base points for readability sake):
\[
\xymatrix{
\rH^p(\pi_1(X),\rH^q(\pi_1(T_{\bar k}),\pi_1(T_{\bar k})))  \ar[r] \ar@{=>}[d] & \rH^p(X,\rH^q(T_{\bar k},\mathbb T(T)))  \ar@{=>}[d]  \\ 
 \rH^{p+q}(\pi_1(E),\pi_1(T_{\bar k}))   \ar[r] &  \rH^{p+q}(E,\mathbb T(T))
}
\] 
The transgression morphisms $d_2^{0,1} : E_2^{0,1}\to E_2^{2,0}$ are thus compatible in the sense that the following diagram commutes:
\[
	\xymatrix{
		\End_{\Gal_k}(\pi_1(T_{\bar k}))\ar[r]^(0.55){\simeq} \ar[d]^{d_2^{0,1}} &\End_k(\mathbb T(T)) \ar[d]^{d_2^{0,1}} \\
\rH^2(\pi_1(X),\pi_1(T_{\bar k}))\ar[r] & \rH^2(X,\mathbb T(T))
	}
\] 
The top arrow is induced by $\pi_1(T_{\bar k}) \simeq \mathbb{T}(T)$, in particular it sends $\id$ to $\id$. 
Now the profinite version of \cite[III, Theorem 4]{hochschild-serre-cohomology} implies that the left vertical map sends $\id$ to $-\pi_1(E/X)$ and a similar argument with Cech cocycles shows that the right vertical map sends $\id$ to $-c_1(E)$.
\end{proof}

\subsection{Killing torsion obstructions} 
\label{sec:kill_tors_obs}

By Proposition~\ref{prop:inj+coincidence} we associate to a $T$-torsor $E\to X$ and a section $s:\Gal_k \to \pi_1(X,\bar x)$ a class 
\[
s^*(c_1(E)) : = s^\ast(\pi_1(E/X)) \in \h^2(k,\mathbb T(T)).
\] 

\begin{lem}
\label{lem:kill_tors_obs}
$s^*(c_1(E))=0$ if and only if $s$ lifts to $\pi_1(E,\bar y)$. 
\end{lem}

\begin{proof}
	This follows from Proposition~\ref{prop:inj+coincidence}.
\end{proof}

We can make this obstruction more tractable thanks to the following lemma.

\begin{lem}
\label{lem:Tate-torsion-free}
The Tate module $\mathbb T(A)$ of an abelian group $A$ is torsion free.	
\end{lem}

\begin{proof}
	This is clear from the expression $\mathbb T(A)=\Hom(\mathbb Q/\mathbb Z, A)$.
\end{proof}

If $E \to X$ is torsion, so is $s^*(c_1(E))$, and Proposition~\ref{prop:Chern-short-exact} together with Lemma~\ref{lem:Tate-torsion-free} show that the obstruction $s^*(c_1(E))$ lives in fact in $\h^1(k,T)$.
\smallskip

\section{Lifting  \texorpdfstring{to $F_D$}{} over the rationals}
\label{sec:lifting}

In this section, $k$ will be a field of characteristic $0$ and $X/k$ a smooth, projective, geometrically connected curve of  genus $\geq 1$. Let $D \subset X$ be an effective reduced Cartier divisor and $U = X\setminus D$. We consider the associated torsors $F_D \to X$ and $E_D \to X$ from Definition~\ref{def_torsors} and study the lifting obstruction for sections $s : \Gal_k \to \pi_1(X,\bar x)$ to the fundamental group of the respective torsors.

\subsection{Vanishing of Brauer obstructions}
The short exact sequence of low degree terms of the Leray spectral sequence for $X \to \Spec(k)$ and $\Gm$ reads
\begin{equation} 
\label{eq:defineBrauerObstructionMap}
0 \to \Pic(X) \to \Pic_{X/k}(k) \xrightarrow{b} \Br(k) \to \Br(X).
\end{equation}
By definition the map $b$ is the Brauer obstruction map with values in the relative Brauer group  $\Br(X/k)$, see Definition~\ref{defi:relBR}, 
that measures the failure of a rational point of the Picard variety to describe an actual line bundle.

\begin{prop}
\label{prop:brauer_obstruction_and_sections}
Let $X/k$ be a smooth projective curve of positive genus such that $\pi_1(X/k)$ admits a section. Then the following holds.
\begin{enumerate}
\item $b(L) = 0$ for all torsion points $L \in \Pic_{X/k}(k)_{\tors}$.
\item If $k/\bQ_p$ is a finite extension, then $\# \Br(X/k)$ is a power of $p$.
\item If $k=\bR$, then $\Br(X/k) = 0$.
\end{enumerate}
\end{prop}

\begin{proof}
Assertion (1) is \cite{stix:periodindex} Proposition 12, and (2) is proven in  \cite{stix:periodindex} Theorem 15. 
Assertion (3) follows  from the real section conjecture, see \cite{stix:habil} \S16.1. Indeed, any section $s : \Gal_\bR \to \pi _1 (X, \bar x)$ comes from a point $x \in X(\bR)$, and evaluation in $x$ yields a retraction to $\Br(k) \to \Br(X)$, showing $\Br(X/k) = 0$.
\end{proof}

\begin{cor} \label{cor:b_vanishes_overQ}
Let $X/\bQ$ be a smooth projective curve of positive genus such that $\pi_1(X/\bQ)$ admits a section.  Then the relative Brauer group 
$\Br(X/\bQ)$ vanishes. 
\end{cor}

\begin{proof}
The Hasse--Brauer--Noether theorem shows that 
\[
\Br(X/\bQ) \inj  \ker\Big(\bigoplus_{v} \Br(X \times_\bQ \bQ_v/\bQ_v)  \xrightarrow{\sum_v \inv_v} \bQ/\bZ\Big)
\]
is injective, where $v$ ranges over all places of $\bQ$. Base change of sections implies that for all places $v$ of $\bQ$ the extension $\pi_1(X \times_\bQ \bQ_v/\bQ_v)$ splits. By Proposition~\ref{prop:brauer_obstruction_and_sections} then $\Br(X\times_\bQ \bR/\bR) = 0$ and $\Br(X \times_\bQ \bQ_p/\bQ_p)$ is cyclic of $p$-power order. This forces $\Br(X/\bQ) = 0$.
\end{proof}

\begin{rem} The proof of Corollary \ref{cor:b_vanishes_overQ} breaks down if one replaces $\bQ$ by any number field $k$, because a general $k$ has different places with the same residue characteristic. This is the point which limits Theorem \ref{thm:mainresult?} to the case of curves over $\bQ$.
\end{rem}

\subsection{Divisibility of line bundles} \label{sec:divinnh}

Let us recall the definition of a neighbourhood of a section.

\begin{defn}
A \textbf{neighbourhood} of a section $s: \Gal_k \to \pi_1(X,\bar x)$ is a connected finite \'etale cover 
\[
h: X' \to X
\]
together with a lift 
\[
s' : \Gal_k \to \pi_1(X',\bar x') \subseteq \pi_1(X,\bar x)
\]
of $s$. A short notation for a neighbourhood is $(X',s')$. 
\end{defn}

Neighbourhoods are geometrically connected over $k$,  because $\pi_1(X',\bar x') \to \Gal_k$ is surjective.

\begin{example} \label{ex:neighbourhoods}
A wealth of neighbourhoods are constructed as follows. Let $\ph : \pi_1(X_{\bar k},\bar x) \surj G$ be a characteristic finite quotient. Then $\ker(\ph)$ is a normal subgroup in $\pi_1(X,\bar x)$ and 
\[
\pi_1(X_\ph,\bar x) = \langle \ker(\ph),  s(\Gal_k) \rangle = \{\gamma s(\sigma) \ ; \ \ph(\gamma) = 1, \ \sigma \in \Gal_k\} \subseteq \pi_1(X,\bar x) 
\]
together with the obvious lift describes a neighbourhood of $s$. Moreover, we have $\pi_1(X_{\ph,\bar k},\bar x) = \ker(\ph)$, so that 
\[
\deg(X_\ph \to X) = \#G.
\]
Since $\pi_1(X_{\bar k},\bar x)$ is topologically finitely generated, the neighbourhoods $X_\ph$ form a cofinal system in the system $X_s = (X')$ of all neighbourhoods.
\end{example}

\begin{prop} 
\label{prop:locdiv}
Let $X/k$ be a smooth projective curve of positive genus. Let $s: \Gal_k \to \pi_1(X,\bar x)$ be a section and let $\dL$ be a line bundle on $X$.  Then the following holds.
\begin{enumerate}
\item For every $n \geq 1$ there is a neighbourhood $(X',s')$ of $s$, such that there is a $M \in \Pic_{X'}(k)$ with 
\[
[\dL|_{X'}] = M^{\otimes n}.
\]
\item If, moreover, for every neighbourhood $(X',s')$ of $s$, the relative Brauer group $\Br(X'/k)$ vanishes, then for every $n \geq 1$ there is a neighbourhood $(X',s')$ of $s$, such that there is a line bundle $\dM$ on $X'$  with 
\[
\dL|_{X'} \simeq \dM^{\otimes n}.
\]
\end{enumerate}
\end{prop}

\begin{defn} \label{defi:locdiv}
If the conclusion (2) of Proposition~\ref{prop:locdiv} holds, then we say that the \textbf{line bundle $\dL$ is divisible locally in neighbourhoods of $s$.}
\end{defn}

\begin{proof}[Proof of Proposition~\ref{prop:locdiv}]
Let $\Pic_{X/k}^{n\ast}$ denote the subgroup of the Picard variety of line bundles of degree divisible by $n$. The boundary map to 
\[
0 \to \Pic_{X/k}[n] \to \Pic_{X/k} \to \Pic_{X/k}^{n\ast} \to 0,
\]
namely
\[
\Pic_{X/k}^{n\ast}(k) \to \rH^1(k,\Pic_{X/k}[n]),
\]
describes the obstruction to being divisible by $n$ in the Picard variety for line bundles of degree divisible by $n$. This obstruction is natural under pullback.

\smallskip

We first prove assertion (1). Since $\pi_1(X_{\bar k},\bar x)$ has finite quotients of order divisible by $n$, we find as in 
Example~\ref{ex:neighbourhoods} a  neighbourhood $(X_1,s_1)$ of $s$ with  $n \mid \deg(X_1/X)$. Then $[\dL|_{X_1}] \in \Pic^{n\ast}_{X_1/k}(k)$ because
\[
\deg(\dL|_{X_{1,\bar k}}) = \deg(X_1/X) \cdot \deg(\dL).
\]

Let $(X_2,s_2)$ be the  neighbourhood of $s_1$ associated to the maximal abelian quotient of exponent $n$ of the group $\pi_1(X_{1,\bar k},\bar x_1)$.
Then the induced map 
\[
\Pic_{X_1/k}[n] = \Hom\big(\pi_1(X_{1,\bar k},\bar x_1),\bZ/n\bZ(1)\big) \to \Hom\big(\pi_1(X_{2,\bar k},\bar x_2),\bZ/n\bZ(1)\big) = \Pic_{X_2/k}[n]
\]
is the zero map. Thus the obstruction for  $[\dL|_{X_1}]$ to divisibility by $n$  in the Picard variety vanishes after restriction to  $X_2$.  This proves (1).

Let $M \in \Pic_{X_2/k}(k)$ be an $n$th root of $[\dL|_{X_2}]$. In order to prove (2) we have to investigate the Brauer obstruction for $M$ to come from an actual line bundle. But this is the class $b(M)$ for the map $b$ in \eqref{eq:defineBrauerObstructionMap} for $X_2/k$ and $b$ vanishes by assumption. This concludes the proof of (2).
\end{proof}

\begin{prop} 
\label{prop:locallydivequalc1killedbys}
Let $X/k$ be a smooth projective curve of positive genus,  let  $s : \Gal_k \to \pi_1(X,\bar x)$ be a Galois section, and let $\dL$ be a line bundle on $X$. Then $\dL$ is locally divisible in neighbourhoods of $s$ if and only if $s^\ast(c_1(\dL)) = 0$.
\end{prop}
\begin{proof}
Let $X_s$ be the projective limit of the pro-system of all neighbourhoods of $s$. Then $\dL$ is locally divisible in neighbourhoods of $s$ if and only if $\dL|_{X_s}$ is divisible in $\Pic(X_s)$. 

The Kummer sequence on $X_s$ yields the exact sequence
\[
\Pic(X_s) \xrightarrow{n \cdot} \Pic(X_s) \xrightarrow{c_1()_n} \rH^2(X_s,\bZ/n\bZ(1)),
\]
so that $\dL|_{X_s}$ is divisible by $n$ on $X_s$ if and only if $c_1(\dL|_{X_s})_n = 0$.  

Naturality of the first Chern class and the isomorphism
\[
\rH^2(X_s,\bZ/n\bZ(1)) = \rH^2(\pi_1(X_s,\bar x),\bZ/n\bZ(1)) \xrightarrow{s^\ast} \rH^2(k,\bZ/n\bZ(1))
\]
show that  $c_1(\dL|_{X_s})_n = 0$  if and only if $s^\ast(c_1(\dL)_n) = 0$. 
Moreover, $s^\ast(c_1(\dL)_n) = 0$ for all $n \geq 1$ if and only if $s^\ast(c_1(\dL)) = 0$, because by Proposition~\ref{prop:Chern-short-exact}~\eqref{item1prop:continuousH2} for $T = \Gm$ and $X = \Spec(k)$ we have
\[
\rH^2(k,\hZ(1)) \simeq \varprojlim_{n \in \bN} \rH^2(k, \bZ/n\bZ(1)). \qedhere
\]
\end{proof}

\begin{prop}
\label{thm:compare_conditions}
Let $X/k$ be a smooth projective curve of positive genus, and let  $s : \Gal_k \to \pi_1(X,\bar x)$ be a Galois section. Consider the following assertions.
\begin{enumerate}
\item[(a)]  All line bundles $\dL$ on $X$ are locally divisible in neighbourhoods of $s$.
\item[(b)] $s^\ast  \circ c_1 : \Pic(X) \to \rH^2(k, \hZ(1))$ vanishes.
\item[(c)] The relative Brauer group $\Br(X/k)$ vanishes.
\item[(c')] The relative Brauer group $\Br(X'/k)$ vanishes for all neighbourhoods $X'$ of $s$.
\end{enumerate}
Then the following implications hold:
\[
(c') \Longrightarrow (a) \iff (b) \Longrightarrow (c).
\]
\end{prop}

\begin{proof}
(c') $\Longrightarrow$ (a) was proven in  Proposition~\ref{prop:locdiv} (2). The equivalence of (a) with (b) follows from 
Proposition~\ref{prop:locallydivequalc1killedbys}.

\smallskip 

For (a) $\Longrightarrow$ (c) we have to show that $b(L) = 0$ for all  $L \in \Pic_{X/k}(k)$.  Since $b(L) \in \Br(X/k)$ is torsion, there is an $n \geq 1$ such that $L^{\otimes n} = [\dM]$ for a line bundle $\dM$ on $X$. By assumption (a), there is a neighbourhood $(X',s')$ of $s$  such that 
$\dM|_{X'}$ admits an $n$th root 
\[
\dM|_{X'} = \dL'^{\otimes n}
\]
with a line bundle $\dL'$  on  $X'$. The difference 
$
\Delta = L|_{X'} - [\dL']
$
is an $n$-torsion element in $\Pic_{X'/k}(k)$.  By Proposition~\ref{prop:brauer_obstruction_and_sections} (1) we compute 
\[
b(L) = b(L|_{X'}) = b(\Delta) + b([\dL']) = 0,
\]
and this proves (c).
\end{proof}

\subsection{Lifting to \texorpdfstring{$F_D$}{FD} over \texorpdfstring{$\bQ$}{}}
In this section we study the lifting problem over the field $\bQ$.

\begin{thm}
\label{thm:mainresult?}
Let $X/\bQ$ be a smooth projective curve of positive genus, and let $D \subset X$ be a union of  torsion sub-packets. Then every Galois section $s : \Gal_\bQ \to \pi_1(X,\bar x)$  lifts to a section $\Gal_\bQ \to \pi_1(F_D,\bar y)$.
\end{thm}

\begin{proof}
Let $D  = \bigcup_{i=1}^n D_i$ be the decomposition into torsion sub-packets $D_i \subset X$.  The torsor $F_D$ is the product (over $X$) of the $F_{D_i}$, and similarly $\pi_1(F_D)$ is the fibre product (over $\pi_1(X)$) of the $\pi_1(F_{D_i})$, so we can assume that $D$ is a single torsion sub-packet.

The section $s$ lifts if and only if $s^\ast(c_1(F_D)) = 0$. Since $T_D$ is the restriction of scalars of $\Gm$, Shapiro's Lemma and Hilbert's Theorem 90 imply $\rH^1(k,T_D) =0$. Using  Proposition~\ref{prop:Chern-short-exact} and Lemma~\ref{lem:Tate-torsion-free}, we find that $s^\ast(c_1(F_D))$ takes values in the torsion free group
\[
\rH^2(k,\bT(T_D)) \simeq \bT\big(\rH^2(k,T_D)\big).
\]
Hence, we may replace $F_D$ by a multiple. The result follows from a diagram chase in 
\[
\xymatrix{ 
&&  \rH^1(X,\mathbb G_m)\ar[r]_{c_1}\ar[dl]_{i_\ast} & \rH^2(X,\widehat{\mathbb Z}(1))\ar[dl] \ar[dd]^{s^*}\\
		&\rH^1(X,T_D)\ar[r]_{c_1}  &\rH^2(X,\bT(T_D))\ar[dd]^{s^*}&\\
 &&&\rH^2(k,\widehat{\mathbb Z}(1))\ar[dl]_{i_\ast} \\
 0\ar[r]&\rH^1(k,T_D)=0\ar[r] &\rH^2(k,\bT(T_D))\ar[r] &\bT(\rH^2(k,T_D))
 }
\]
Namely, since $D$ is a torsion sub-packet, we can apply Proposition~\ref{thm:crit-torsion-torsors} and choose an integer $N\geq 1$ so that $N\cdot E_D$ is trivial. Then $N\cdot F_D$ comes from a $\mathbb G_m$-torsor, and the result follows from 
Proposition~\ref{thm:compare_conditions} and Corollary~\ref{cor:b_vanishes_overQ}.
\end{proof}

\begin{cor}
\label{cor:liftoutsideQpoints}
Let $X/\bQ$ be a smooth projective curve of positive genus, and let $U = X \setminus X(\bQ)$ be the complement of the set of all $\bQ$-rational points.  Then every Galois section $s : \Gal_\bQ \to \pi_1(X,\bar x)$  lifts to a section $\Gal_\bQ \to \pi_1^{cc}(U,\bar x)$.
\end{cor}
\begin{proof}
If $X(\bQ)$ is infinite (at most for $X$ of genus $1$), then we understand $\pi_1^{cc}(U,\bar x)$ as the natural projective limit of $\pi_1^{cc}(X \setminus D,\bar x)$ with $D$ ranging over all finite subsets $D \subseteq X(\bQ)$. We may therefore restrict to the case of $U = X \setminus D$ and $D \subseteq X(\bQ)$ a finite set. 

The divisor $D \subset X$ is a union of torsion sub-packets. Therefore the section $s$ lifts to a section $\Gal_\bQ \to \pi_1(F_D,\bar y)$ by Theorem~\ref{thm:mainresult?}, and since $\pi_1^{cc}(U,\bar x) \simeq \pi_1(F_D,\bar y)$ by Proposition~\ref{cc}, this  completes the proof.
\end{proof}

\begin{rem} Note that similar arguments enable to show over a local field $k$ a prime to $p$ version of Theorem \ref{thm:mainresult?}, with $\pi_1(F_D,\bar y)$ replaced by its quotient by the $p$-part of the geometric fundamental group $\pi _1 ((T_{D})_{\bar k} )$.
\end{rem}


\bibliography{lifting_short-mn}

\end{document}